\documentclass{amsart}
\usepackage{amsmath}
\usepackage[abbrev]{amsrefs}
\usepackage{amsfonts}
\usepackage{amsthm}
\usepackage{tikz}
\usepackage[english]{babel}
\usepackage[a4paper, total={6in, 8in}]{geometry}
\usepackage{color}
\usepackage{bbm}
\usepackage{hyperref}
\usepackage{mathrsfs}
\usepackage{graphicx}

\newcommand{\odv}[2]{\frac{\mathrm{d}#1}{\mathrm{d}#2}}

\def \tC{\widetilde{C}}

\def \pA{\mathcal{A}}
\def \pE{\mathcal{E}}

\def \R{\mathbb{R}}
\def \P{\mathbb{P}}
\def \E{\mathbb{E}}

\newtheorem{theorem}{Theorem}[section]
\newtheorem{lemma}[theorem]{Lemma}
\newtheorem{proposal}[theorem]{Proposition}
\newtheorem{fact}[theorem]{Fact}
\newtheorem{definition}[theorem]{Definition}
\newtheorem{corollary}[theorem]{Corollary}
\newtheorem{remark}{Remark}

\begin{document}

\title{Ground-state representation for fractional Laplacian on half-line}
\author{Tomasz Jakubowski}
\address{Wroc\l{}aw University of Science and Technology}
\email{tomasz.jakubowski@pwr.edu.pl}
\thanks{The paper was partially supported  by the NCN grant 2015/18/E/ST1/00239}
\author{Pawe\l{} Maciocha}
\address{Wroc\l{}aw University of Science and Technology}
\email{pawel.maciocha@pwr.edu.pl}


\begin{abstract}
We give ground-state representation for the fractional Laplacian with Dirichlet condition on the half-line.

\end{abstract}

\maketitle

	\section{Introduction}
	In the paper \cite{MR3460023}, K. Bogdan et al. proposed a general method of constructing supermedian functions for semigroups. 
This approach was applied in \cite{MR3933622} and \cite{MR4140086} to study singular Schr\"odinger perturbation of fractional Laplacian. 
In this paper we apply this methodology to the Dirichlet kernel of the half-line $(0,\infty)$ for the fractional Laplacian and we obtain a wide spectrum of the ground state representations of the corresponding quadratic form. In doing so we resolve mayor technical problems related to compensation of kernels and divergent integrals.

	We  write '':='' to indicate definitions, e.g., $a \wedge b := \min\{a,b\}$ and $a \vee b := \max\{a,b\}.$ We write $f(x) \approx g(x)$ if $f, g \geq 0$ and $c^{-1} g(x) \leq f(x) \leq c g(x) $ for some positive number $c>0$ and all the arguments $x$. 
	
	\subsection{Fractional Laplacian and $\alpha$-stable L\'evy process}
	Let $\alpha \in (0,2)$, 
	\begin{align*}
		\pA_{\alpha} = \frac{\alpha \Gamma(\alpha)\sin\big(\pi\frac{\alpha}{2}\big)}{\pi} \qquad
		\mbox{and} \qquad
		\nu(y) = \pA_\alpha |y|^{-1-\alpha}.
	\end{align*}
	For (smooth and compactly supported) $\phi \in C_c^\infty(\R)$, the fractional Laplacian is
	\begin{align*}
		\Delta^{\alpha/2} \phi(x) = \lim_{\varepsilon\downarrow 0} \int_{B(0,\varepsilon)^c}(\phi(x+y)-\phi(x)) \nu(y) dy, \quad x \in \R.
	\end{align*}
	Let $p_t$ be the smooth real-valued function on $\R$ with Fourier transform
	\begin{align} \label{eq:ftp_t} 
		\int_{\R} p_t(x) e^{ix\xi} dx = e^{-t|\xi|^\alpha}, \qquad t>0,\, \xi \in \R.
	\end{align}
	According to the L\'evy-Khinchine formula, $p_t$ is a density of the probabilistic convolution semigroup with L\'evy measure $\nu(y)dy$, see e.g. \cite{MR2569321}. Let
	\begin{align*}
		p(t,x,y) = p_t(y-x).
	\end{align*}
	We consider the time-homogeneous transition probability $(t,x,A) \mapsto \int_Ap(t,x,y) dy$, $t>0$, $x\in \R$, $A \subset \R$. By Kolmogorov’s and Dynkin-Kinney’s theorems there is a stochastic process $(X_t, \P^x)$ with c\'adl\'ag paths and initial distribution $\P^x(X(0)=x)=1$.
	We denote by $\P^x$ and $\E^x$ the distribution and expectation for the process starting at $x$. 
	We call $X_t$ the isotropic $\alpha$-stable process with index of stability $\alpha \in (0,2)$. In fact, $X_t$ is a L\'evy process with zero Gaussian part and drift, and with L\'evy measure $\nu(y)dy$. 
	
	\indent From (\ref{eq:ftp_t}), we have the scaling property
	\begin{equation*}
		p_t(x) = t^{-1/\alpha} p_1(t^{-1/\alpha} x).
	\end{equation*}
	It is well known that (see e.g. \cite{MR2569321})
	\begin{equation*}
		p_t(x) \approx t^{-1/\alpha} \wedge \frac{t}{|x|^{1+\alpha}},\qquad t>0,~x \in \mathbb{R}.
	\end{equation*}
	Additionally, the function $p(t,x,y)$ satisfies the Chapman-Kolmogorov equation 
	\begin{equation*}
		p(t + s,x,y) = \int_{ \mathbb{R} } p(t,x,z)p(s,z,y)dz, \qquad s,t>0,~x,y \in \mathbb{R}.
	\end{equation*}
	\subsection{Killed process}
	
Throughout the paper we let $D = (0, \infty) \subset \mathbb{R}$. We define the time of the first exit of the process $X_t$ from $D$ by 
	\begin{equation*}
		\tau_D = \inf \{t\geq0: X_t \in D^c \}.
	\end{equation*}
The	random variable $\tau_D$ is almost surely finite Markov moment (see \eqref{eq:lim_prob_tau_D}). 
	By $P_t^D$ we denote the semigroup generated by the process $X_t$ killed on exiting $D$. The semigroup is determined by transition densities $p_D(t,x,y)$ given by the Hunt formula (see e.g. \cite{MR1329992})
	\begin{equation} \label{eq:Hunt}
		p_D(t,x,y) = p(t,x,y) - \mathbb{E}^x[ p(t-\tau_D, X_{  \tau_D}, y) \mathbf{1}_{ \{ \tau_D < t \} } ],\qquad t>0,~x,y\in D.
	\end{equation}
	It well known that the density $p_D(t,x,y)$ is symmetric, i.e.
	$p_D(t,x,y) = p_D(t,y,x)$, continuous in $(t,x,y)$ for $t>0$, $x,y 
	\in D$, satisfies the Chapman-Kolmogorov equation 
	\begin{equation*}
		p_D(t + s,x,y) = \int_D p_D(t,x,z)p_D(s,z,y) dz,
	\end{equation*}
	and for any nonnegative Borel functions, we have
	\begin{align}\label{eq:pDdef}
		P^D_t f(x) := \int_D p_D(t,x,y) f(y) dy = \E^x\left(\mathbbm{1}_{\{t<\tau_D\}} f(X_t) \right), \qquad t>0, \, x \in D.
	\end{align}
	We note that $p_D$ admits the same scaling as $p$,
	\begin{equation} \label{eq:sca}
		p_D(t,x,y) = t^{-1/\alpha} p_D(1,t^{-1/\alpha} y,t^{-1/\alpha} x).
	\end{equation}
	It is known that (see e.g. \cite{MR2722789})
	\begin{equation} \label{eq:oppt}
		p_D(t,x,y) \approx p(t,x,y)\bigg( 1 \wedge \frac{ x^{\alpha/2} }{ t^{1/2} } \bigg) \bigg( 1 \wedge \frac{ y^{\alpha/2} }{ t^{1/2} } \bigg),\qquad t>0,~x,y\in D,
	\end{equation} 
	while for $x \not\in D$ or $y\not\in D$, we have $p_D(t,x,y)=0$, $t>0$. By \eqref{eq:pDdef} and integrating the Hunt formula (\ref{eq:Hunt}), we get
	\begin{equation} \label{eq:prob_tau_D}
		\mathbb{P}^{x} (\tau_{D} > t)= \int_D p_D(t,x,y)dy.
	\end{equation}
	
	\subsection{Main results}
	Recall that $D=(0,\infty)$ and $\alpha \in (0,2)$.	Consider  $\beta \in (0,1),~ \gamma \in (\beta + \alpha/2, 1 + \alpha/2)$	and define
	\begin{equation}\label{eq:constC}
		\mathcal{C} = \int_{0}^{\infty} \int_{D} p_D(t,1,y) t^{ (-\alpha/2 - \beta + \gamma)/\alpha } y^{-\gamma} dydt
	\end{equation}
and
\begin{equation*}
	f(t) = 
	\begin{cases}
	\mathcal{C}^{-1} t^{ (-\alpha/2 - \beta + \gamma)/\alpha },\qquad &\text{for}~t>0,\\
			0,\qquad &\text{for}~t\leq0.
	\end{cases}
\end{equation*}
	According to \cite{MR3460023} we define
	\begin{align*}
		h_{\beta}(x) &= \int_{0}^{\infty} \int_{D} p_D(t,x,y)f(t)y^{-\gamma} dydt,\qquad x \in D, \\
		q_{\beta}(x) &= \frac{1}{h_{\beta}(x)} \int_{0}^{\infty} \int_{D} p_D(t,x,y)f'(t)y^{-\gamma} dydt,\qquad x \in D.
	\end{align*}
By scaling property \eqref{eq:sca}, $h_\beta = x^{\alpha/2-\beta}$ (see Lemma \ref{l:hbeta}). It turns out that $q_\beta$  does not depend on $\gamma$ as well. In the previous papers \cite{MR3933622}, \cite{MR4140086} this construction was used to the convolution semigroups with the Dirac measure $\delta_0$ in the place of $\mu(dy) = y^{-\gamma} dy$. In our case such a choice is impossible because $p_D(t,x,y)=0$ for $y \le0$. Our approach shows how the construction introduced in \cite{MR3460023} may be used for more general semigroup than convolution ones, provided one can find the proper measure $\mu$.   Our main results are stated in the following theorems. 
	\begin{theorem}\label{tw:qbeta}
		Let $\beta \in (0,1)$, then
		\begin{align*}
			q_\beta(x) = \kappa_\beta x^{-\alpha}, \qquad x\in D,
		\end{align*} 
		where
		\begin{align}\label{eq:kappabeta}
			\kappa_\beta = \frac{\Gamma(\beta+\alpha/2)\Gamma(1-\beta+\alpha/2)}{\Gamma(\beta)\Gamma(1-\beta)}.
		\end{align}
	\end{theorem}
	The main difficulty here is to obtain the exact value of the constant $\kappa_\beta$. In contrary to the case of the free process considered in \cite{MR3460023}, generally we cannot calculate the constant $\mathcal{C}$ in \eqref{eq:constC} for arbitrary $\gamma$. One may try here the approach involving the Mellin transform of supremum process and calculate $\mathcal{C}$ for $\gamma=0$ (see Remark \ref{rem:1}). However, it leads to very complicated formulas, so we choose  another method to prove Theorem \ref{tw:qbeta} by finding the value $\mathcal{C}$ for $\gamma = \beta+\alpha/2$. This result is stated in the following theorem.
	\begin{theorem} \label{thm:1}
		Let $\beta \in (0,1)$. Then, 
		\begin{equation} \label{eq:1}
			\int_0^\infty \int_D p_D(t,x,y) y^{-\beta-\alpha/2} dy dt = \kappa^{-1}_{ \beta } x^{\alpha/2 -\beta},\qquad x \in D,
		\end{equation}
		where $\kappa_\beta$ is given by \eqref{eq:kappabeta}.
	\end{theorem} 
\noindent 	According to \eqref{eq:pDdef}, in the probabilistic terms, the statement of Theorem \ref{thm:1} reads as follows:
	\begin{align*}
		\E^x\left(\int_0^{\tau_D} X_t^{-\beta-\alpha/2} dt \right) = \kappa^{-1}_{ \beta } x^{\alpha/2 -\beta},\qquad x \in D.
	\end{align*}

	As an application of Theorem \ref{tw:qbeta} we prove the Hardy identity for the Dirichlet form 
	\begin{equation*}
		\mathcal{E}_D(u,u) = \lim_{t \to 0+}\frac{1}{t} (u-P_t^Du,u),\qquad u\in L^2(D).
	\end{equation*}
	The subject of Hardy identities and inequalities was initiated in 1920, when Hardy \cite{MR1544414} discovered that
	\begin{equation}\label{eq:hardycls}
		\int_{0}^{\infty} \left[u'(x)\right]^2 dx \geq \frac{1}{4} \int_0^{\infty} \frac{u(x)^2}{x^2} dx,
	\end{equation}
	for absolutely continuous functions $u$ such that $u(0)=0$ and $u' \in L^2(0,\infty)$. Later, many generalizations of  \eqref{eq:hardycls} were proven, where the lefthand side of \eqref{eq:hardycls}  was  replaced by various symmetric Dirichlet forms $\mathcal{E}$ in the sense of Fukushima, Oshima, Takeda \cite{MR2778606}. In particular, for all 
	$d\ge1$, $0<\alpha<d \wedge 2$, $0 \leq \beta \leq d-\alpha$,
	and $u \in L^2(\R^d)$, the following Hardy-type identity holds (see \cite{MR2425175} and \cite{MR3460023})
	\begin{align}\label{eq:p2}
		\mathcal{E}(u,u)= 
		\kappa_{\beta}^{\R^d}
		\int_{\R^d} \frac{u(x)^2}{|x|^\alpha}\,dx
		+ \frac{1}{2} \int_{\R^d}\!\int_{\R^d}
		\left[\frac{u(x)}{|x|^{-\beta}}-\frac{u(y)}{|y|^{-\beta}}\right]^2
		|x|^{-\beta}|y|^{-\beta} \nu(x,y)
		\,dy\,dx,
	\end{align}
	where  $\mathcal{E}(u,u) = \lim_{t \to 0+}\frac{1}{t} (u-P_t u,u)$, $P_t$ is the stable semigroup on $\R^d$ and
	\begin{align}\label{eq:kappabRd}
		\kappa_\beta^{\R^d} = \frac{2^\alpha \Gamma((\beta+\alpha)/2)\Gamma((d-\beta)/2)}{\Gamma(\beta/2)\Gamma((d-\alpha-\beta)/2)}, \qquad 0< \beta < d-\alpha.
	\end{align}
	The identity \eqref{eq:p2} is also called a ground state representation. It yields the Hardy type inequality for $\mathcal{E}$ (see also \cite{MR0436854}, \cite{MR1254832} and \cite{MR1717839})
	\begin{align}\label{eq:HardyIneqRd}
		\mathcal{E}(u,u)\ge 
		\kappa_{\beta}
		\int_{\R^d} \frac{u(x)^2}{|x|^\alpha}\,dx.
	\end{align}
	
	We get the analogous result to \eqref{eq:p2} for the form $\mathcal{E}_D$.
	\begin{theorem}\label{thm:HardyId}
		If $u \in L^2(D),$ then 
		\begin{align}\label{eq:HardyId}
			\mathcal{E}_D(u,u) = \kappa_{ \beta } \int_D \frac{u^2(x)}{ x^{\alpha} } dx + \frac{1}{2} \int_D \int_D \bigg( \frac{u(x)}{x^{		\alpha/2 - \beta}} - \frac{u(y)}{y^{\alpha/2 - \beta}} \bigg)^2 x^{\alpha/2 - \beta} y^{\alpha/2 - \beta} \nu(x-y) dxdy.
		\end{align}
	\end{theorem}
As a corollary, in Proposition \ref{Prop:HardyIn} we give a new proof of the Hardy inequality for $\pE_D$ obtained in \cite{MR2663757} (see also \cite{210901704} for its generalization to Sobolev-Bregman forms). For $\beta=1/2$, the identity \eqref{eq:HardyId} is a special case of \cite[Theorem 1.2]{MR2723817} with  $p=2$ and $N=1$.  We note that the formula \eqref{eq:HardyId} may be derived from \eqref{eq:p2} with $d=1$ by taking $u$ with support in $D$, see Remark \ref{rem:3}. However, the range of arguments in this case is limited to $\alpha \in (0,1)$ and $\beta \in (\alpha/2,1-\alpha/2)$. Note that in our approach this case is the easiest one (see e.g. Fact \ref{f.2}) and the generalization to the whole range of parameters $\alpha \in(0,2)$ and $\beta \in(0,1)$ is non-trivial.

\subsection{Further discussion}
	
	If we compare our result for $\alpha \in (0,2)$ with the case $\alpha=2$ and $d=1$, we may observe several similarities. First, note that for $\beta \in (0,1)$ and sufficiently regular $u$ we have the following Hardy identity (see e.g. \cite[Theorem 3]{MR4388136} for the special case $\beta=1/2$)
\begin{align*}
\int_0^\infty [u'(x)]^2 dx = \beta(1-\beta) \int_0^\infty \frac{u^2(x)}{x^2} dx + \int_0^\infty \left[x^{1-\beta} (x^{\beta-1}u(x))'\right]^2 dx.	
\end{align*}
Hence, \eqref{eq:HardyId} may be considered as a fractional counterpart of the formula above.	Additionally, if we put $\alpha=2$ in \eqref{eq:kappabeta}, we get $\kappa_{\beta} = \beta(1-\beta)$. Both, for $\alpha\in(0,2)$ and $\alpha=2$,  $\kappa_\beta$ attains its maximum for $\beta = 1/2$ (see Figure \ref{fig:k}) and $\kappa_{1/2} = \Gamma((\alpha+1)/2)^2/\pi$ is the best constant in the Hardy inequality.  
	
The subject of this paper is also strictly connected with Schr\"odinger perturbations of semigroups by Hardy potential, see e.g. \cites{MR3479207, MR3492734, MR3933622, MR4140086, MR4216547, MR2425175, MR4163128,MR4206613}. 	In particular, in \cite{MR3933622}  the authors studied the Schr\"odinger perturbations of fractional Laplacian in $\R^d$  by $\kappa |x|^{-\alpha}$ and obtained the sharp estimates of the perturbed semigroup. It turns out that the critical value of $\kappa$ for which the perturbed density is finite coincides with the best constant in the Hardy inequality \eqref{eq:HardyIneqRd}. 

We point out that Theorem \ref{tw:qbeta} may be the starting point for further study of Schr\"odinger perturbation of $p_D$ by the potential $\kappa x^{-\alpha}$. Here, we may expect that like in \cite{MR742415} and \cite{MR3933622} the best constant $\kappa_{1/2}$ in Hardy inequality \eqref{Eq:HardyIn} is also critical in the sense that for $\kappa > \kappa_{1/2}$ we will have instantaneous blow-up of the perturbed heat kernel $\tilde{p}$. To get this result probably one has to get the proper estimates of $\tilde{p}$ (see e.g. \cite{MR3933622}), which we postpone to a forthcoming paper.

	The last remark concerns the range of parameters $\alpha$ and $\beta$. We note that like in \eqref{eq:p2} the authors in \cite{MR3933622} assumed  that $\alpha <d$, which for $d=1$ gives the restriction $\alpha<1$. This condition was imposed to get the integrability of the potential $|x|^{-\alpha}$ at $0$. However, in our case by considering the process killed on exiting the halfline, the additional decay at $0$ of the kernel $p_D$ permits to study the perturbations by $\kappa x^{-\alpha}$ for the whole range of $\alpha \in (0,2)$. Moreover, since $p_D(t,x,y) < p(t,x,y)$, potentially the bigger perturbation may be considered also for $\alpha<1$. In Remark \ref{rem:2}, we show that for $\beta \in (0, (1-\alpha)/2)$, $\kappa_\beta > \kappa_\beta^\R$, where $\kappa_\beta^\R$ is given by \eqref{eq:kappabRd} (see Figure \ref{fig:2}).

	The paper is organized as follows. In Section \ref{sec:2} we calculate some auxiliary integrals involving gamma functions and give some definitions from potential theory. In Section \ref{sec:3} we prove Theorems \ref{tw:qbeta} and \ref{thm:1}. In Section \ref{sec:4} we apply Theorem \ref{tw:qbeta} to study Hardy identities for the quadratic form connected with the semigroup $P_t^D$.

	\section{Preliminaries}\label{sec:2}
	\subsection{Gamma and Beta functions}
	For $x>0$, we define the gamma function
	\begin{equation*}
		\Gamma(x) = \int_{0}^{\infty} t^{x-1} e^{-t} dt.
	\end{equation*}
	
	\noindent By using the property $x\Gamma(x) = \Gamma(x+1)$
	we extend the definition of $\Gamma(x)$ to negative non-integer values of $x$.
	The beta function is defined by
	\begin{equation}\label{eq:beta_def}
		B(x,y) = \frac{ \Gamma(x) \Gamma(y) }{ \Gamma(x+y) },
	\end{equation}
	for $x,y \in \mathbb{R} \setminus \{0,-1, -2, \dots\}$.
	Recall that
	\begin{align} \label{eq:beta_def_2}
		B(x,y) &= \int_{0}^{1} t^{x-1} (1-t)^{y-1} dt = \int_{0}^{\infty} \frac{t^{x-1}}{ (1+t)^{x+y} }dt,  \qquad x,y >0.
	\end{align}

	\begin{lemma} \label{l:int_as_beta_plus_c}
		Let $a \in (-1,0)$ and $b>0$. We have
		\begin{equation*}
			\int_0^1 s^{a-1} \big( (1-s)^{b-1} - 1 \big)ds = - \frac{1}{a} + B(a,b).
		\end{equation*}
	\end{lemma}
	\begin{proof}
		Note that $\lim_{s\to0} s^{-1}\big( (1-s)^{b-1} - 1 \big) = 1-b $. Hence,
		using the integration by parts, \eqref{eq:beta_def_2} 
		and \eqref{eq:beta_def}, we get 
		\begin{align*}
			a \int_0^1 s^{a-1} \big( (1-s)^{b-1} - 1 \big)ds &= a \int_0^1 (s^{a-1}(1-s) + s^a) \big( (1-s)^{b-1} - 1 \big)ds \\
			&= \int_0^1 as^{a-1} \big( (1-s)^{b} - (1-s) \big)ds + a \int_0^1 s^a \big( (1-s)^{b-1} - 1 \big)ds \\
			&= \int_0^1 s^{a} \big( b(1-s)^{b-1} - 1 \big)ds + a \int_0^1 s^a \big( (1-s)^{b-1} - 1 \big)ds \\		
			&= (a+b) B(a+1,b) -\frac{1}{a+1}-\frac{a}{a+1} = a B(a,b) -1.		
		\end{align*}
	\end{proof}

	\noindent For $\beta \in (0,1)$ and $\alpha \in (0,2)$ such that $\beta + \alpha/2 \not=1$ and $\beta-\alpha/2 \not=0$ we define
	\begin{align} 
		\widehat{C}_{\alpha, \beta} &= B(1-\beta - \alpha/2, \alpha) + B(\alpha, \beta - \alpha/2), \label{const.1} \\
		\widetilde{C}_{\alpha, \beta} &= B(1-\beta - \alpha/2, \beta - \alpha/2). \label{const.2}
	\end{align}
	
	\begin{fact} \label{f.2}
		For $\alpha \in (0,1)$ and $\beta \in (\alpha/2, 1 - \alpha/2)$, we have
		\begin{align}
			\int_D |1 - w|^{\alpha - 1} w^{-\beta - \alpha/2} dw = \widehat{C}_{\alpha, \beta}, \label{eq1:f.2}\\
			\int_D (1 + w)^{\alpha - 1} w^{-\beta - \alpha/2} dw = \tC_{\alpha, \beta}. \label{eq2:f.2}
		\end{align}
		
	\end{fact}
	\begin{proof}
		By \eqref{eq:beta_def} and \ref{eq:beta_def_2}, we have
		\begin{align*}
			\int_D |1 - w|^{\alpha - 1} w^{-\beta - \alpha/2} dw &= \int_0^1 (1-w)^{\alpha -1} w^{-\beta - \alpha/2} dw + \int_0^{\infty} w^{\alpha - 1} (w+1)^{-\beta - \alpha/2} dw \\ &= B(1 - \beta - \alpha/2, \alpha) + B(\alpha, \beta - \alpha/2).
		\end{align*}
		The second equality follows directly from \eqref{eq:beta_def_2}.
	\end{proof}
	\noindent 	We will also need similar equalities for wider range of parameters $\alpha$ and $\beta$. To secure convergence of integrals we compensate the expressions $|1-w|^{\alpha-1}$ and $(1+w)^{\alpha-1}$, which improves the decay near 0 and 1.

	\begin{fact} \label{f.1}
		In the cases of $\alpha \in (0,1)$, $\beta \in (0, \alpha/2)$ and $\alpha \in (1,2),$ $\beta \in(0,1-\alpha/2)$, we have
		\begin{equation} \label{eq1:f.1}
			\int_D \big( |1 - w|^{\alpha - 1} - w^{\alpha - 1} \big)w^{-\beta - \alpha/2} dw = \widehat{C}_{\alpha, \beta}, 
		\end{equation}
		and
		\begin{equation} \label{eq2:f.1}
			\int_D \big( (1 + w)^{\alpha - 1} - w^{\alpha - 1} \big)w^{-\beta - \alpha/2} dw = \widetilde{C}_{\alpha, \beta}. 
		\end{equation}
	\end{fact}
	\begin{proof}
		We have 
		\begin{align*}
			&\int_{D} \big( |1-w|^{\alpha - 1} - w^{\alpha - 1} \big) w^{-\beta - \alpha/2} dw \\ &= \int_{0}^{1} \big(  (1 - w)^{\alpha - 1} - w^{\alpha - 1} \big) w^{-\beta - \alpha/2} dw + \int_{0}^{\infty} \big( w^{\alpha - 1} - (w + 1)^{\alpha - 1} \big) (w + 1)^{-\beta - \alpha/2} dw.
		\end{align*}
		By (\ref{eq:beta_def_2}), the first integral is equal to 
		\begin{equation*}
			\int_{0}^{1} \big( (1 - w)^{\alpha - 1} - w^{\alpha - 1} \big) w^{-\beta - \alpha/2} dw  = B(1 - \beta - \alpha/2, \alpha) - \frac{1}{\alpha/2 - \beta}.
		\end{equation*}
		Let $w+1 = s^{-1}$. By the Fubini's theorem and Lemma \ref{l:int_as_beta_plus_c}, we get
		\begin{align*}
			&\int_{0}^{\infty} \big( w^{\alpha - 1} - (w + 1)^{\alpha - 1} \big) (w + 1)^{-\beta - \alpha/2} dw = \int_{ 0 }^{ 1 } \bigg( \Big( \frac{s}{1-s} \Big)^{1-\alpha} - s^{1 - \alpha}  \bigg)s^{\beta + \alpha/2 - 2} ds \\ &= \int_{0}^{1} s^{\beta - 1 - \alpha/2} \big( (1-s)^{\alpha - 1} - 1 \big) ds  = \frac{1}{\alpha/2 - \beta} + B(\alpha, \beta - \alpha/2),
		\end{align*}
		which proves (\ref{eq1:f.1}). Next, by Fubini's theorem and again by (\ref{eq:beta_def_2}),
		\begin{align*}
			&\int_{0}^{\infty} \big( (1 + w)^{\alpha - 1} - w^{\alpha - 1} \big) w^{-\beta - \alpha/2} dw = \int_{0}^{\infty} ( \alpha - 1 ) \bigg( \int_{0}^{1} (w+t)^{\alpha -2} dt \bigg) w^{-\beta - \alpha/2} dw \\ &= ( \alpha - 1 ) \int_{0}^{1} t^{\alpha/2 - 1 - \beta} \int_{0}^{\infty} (1 + z)^{\alpha - 2} z^{-\beta - \alpha/2} dz dt \\ &= (\alpha - 1) B(1-\beta - \alpha/2, 1 + \beta - \alpha/2)  \int_{0}^{1} t^{\alpha/2 - 1 - \beta} dt = B(1-\beta - \alpha/2, \beta - \alpha/2),
		\end{align*}
		which proves (\ref{eq2:f.1}). 
	\end{proof}

	\begin{fact} \label{f.4}
		Let $\alpha \in (1,2)$ and $\beta \in (1 - \alpha/2, 1)$. We have
		\begin{equation*}
			\int_{D} (|1 - w|^{\alpha - 1} - 1)w^{-\beta - \alpha/2} dw = \widehat{C}_{\alpha, \beta},
		\end{equation*}
		and
		\begin{equation*}
			\int_{D} ((1 + w)^{\alpha - 1} - 1)w^{-\beta - \alpha/2} dw = \widetilde{C}_{\alpha, \beta}.
		\end{equation*}
	\end{fact}
	\begin{proof}
		By Lemma \ref{l:int_as_beta_plus_c},
		\begin{align*}
			&\int_{0}^{1} ( |1 - w|^{\alpha - 1} - 1)w^{-\beta - \alpha/2} dw = B(1-\beta - \alpha/2, \alpha) + \frac{1}{\beta + \alpha/2 - 1}.
		\end{align*}
		Next, by (\ref{eq:beta_def_2})
		\begin{align*}
			&\int_{1}^{\infty} ( |w-1|^{\alpha - 1} - 1) w^{-\beta - \alpha/2} dw = \int_{0}^{\infty} ( w^{\alpha - 1} - 1)(w+1)^{-\beta - \alpha/2} dw \\&= \int_{0}^{\infty} w^{\alpha - 1}(w+1)^{-\beta - \alpha/2} dw - \int_{0}^{\infty} (w+1)^{-\beta - \alpha/2} dw \\&= B(\alpha, \beta - \alpha/2) - \frac{1}{\beta + \alpha/2 - 1}.
		\end{align*}
		Hence we get the first equality in the lemma assertion. We also get the second equality by (\ref{eq:beta_def_2}):
		\begin{align*}
			&\int_{D} ( (1 + w)^{\alpha - 1} - 1)w^{-\beta - \alpha/2} dw =  \int_{D} \Big( (1-\alpha) \int_{1}^{1+w} y^{\alpha -2} dy \Big) w^{-\beta - \alpha/2} dw \\&= (1-\alpha) \int_{1}^{\infty} \int_{y-1}^{\infty} w^{-\beta - \alpha/2} y^{\alpha-2} dw dy = (1-\alpha) \int_{1}^{\infty}  \frac{1}{1 -\beta - \alpha/2} (y-1)^{1 - \beta - \alpha/2} y^{\alpha-2} dy \\&= \frac{1-\alpha}{ 1 - \beta - \alpha/2 } \int_{0}^{\infty} y^{1-\beta - \alpha/2} (y+1)^{\alpha-2} dy = \frac{1-\alpha}{1 - \beta - \alpha/2 } B(2 - \beta - \alpha/2, \beta - \alpha/2). 
		\end{align*}
	\end{proof}
	The last fact concerns the case $\alpha=1$.
	
	\begin{fact} \label{f.5}
		Let $\beta \in (0,1/2)$. We have
		\begin{equation*}
			\int_{0}^{\infty} y^{-\beta - 1/2} \ln(|1-y|/y)dy = \frac{ \pi \sin(\pi \beta ) }{ (1/2 - \beta)\cos( \pi \beta ) },
		\end{equation*}
		and 
		\begin{equation*}
			\int_{0}^{\infty} y^{-\beta - 1/2} \ln(|1+y|/y)dy = \frac{\pi}{ (1/2 - \beta) \cos(\pi \beta) }.
		\end{equation*}
	\end{fact}
	
	\begin{proof}
		Substituting $x=1/y$ and applying \cite[Equation 4.293.7]{MR2360010}, we obtain
		\begin{equation*}
			\int_{0}^{\infty} y^{-\beta - 1/2} \ln(|1-y|/y)dy  = \int_{0}^{\infty} x^{ (\beta - 1/2) - 1} \ln(|1-x|)dx = \frac{ \pi \sin(\pi \beta ) }{ (1/2 - \beta)\cos( \pi \beta ) }.
		\end{equation*} 
		By making the same substitution and by \cite[Equation 4.293.10]{MR2360010}, we get 
		\begin{equation*}
			\int_{0}^{\infty} y^{-\beta - 1/2} \ln(|1+y|/y)dy  = \int_{0}^{\infty} x^{ (\beta - 1/2) - 1} \ln(|1+x|)dx = \frac{\pi}{ (1/2 - \beta) \cos(\pi \beta) }.
		\end{equation*}
	\end{proof}
	
	\subsection{Green function and Poisson Kernel}
	We recall some facts from potential theory of stable processes, see e.g. \cite{MR1825645}. 
	For $\alpha<1$ the process $X_t$ is transient and its potential kernel is given by
	\begin{align*}
		K_\alpha(x) = \int_0^\infty p_t(x) dt.
	\end{align*}
	For $\alpha \ge 1$, $X_t$ is recurrent and we consider the compensated kernels
	\begin{align*}
		K_\alpha(x) = \int_0^\infty (p(t,x) - p(t,x_0))dt,
	\end{align*}
	where $x_0 =0$ for $\alpha>1$ and $x_0=1$ for $\alpha=1$. It turns out that $K_\alpha$ is radial and  
	\begin{align*}
		K_\alpha(x)=
		\begin{cases}
			C_\alpha |x|^{\alpha-1} & \mbox{for\; } \alpha\not=1,\\
			-\frac{1}{\pi} \ln(|x|) & \mbox{for\; } \alpha=1,
		\end{cases}
	\end{align*}
	where 
	\begin{align}\label{C_alpha}
		C_\alpha = \frac{1}{2\Gamma(\alpha) \cos(\pi\alpha/2)}.
	\end{align}
	We note that $C_{-\alpha} = \pA_{\alpha}.$
	Recall that $D = (0, \infty)$. 
	\begin{definition}
		We define Green's function of the set $D$
		\begin{equation}\label{def:GreenF}
			G_D(x,y) = \int_{0}^{\infty} p_D(t,x,y)dt,\qquad ~x,y \in D.
		\end{equation}
	\end{definition}
	\noindent $G_D(x,y)$ is symmetric, continuous on $D \times D$ with values in $[0,\infty].$ If at least one of the arguments belongs to $D^c$, then $G_D(x,y) = 0 $. Moreover for $x\neq y \in D$, $G_D(x,y) < \infty$ (see \cite{MR2722789}). By integrating the Hunt formula we obtain
	\begin{align}\label{eq:GreenHunt}
		G_D(x,y) = K_\alpha(x,y) - \E^xK_\alpha(X_{\tau_D},y) = K_\alpha(x,y) - \E^yK_\alpha(X_{\tau_D},x), 
	\end{align}
	when the last equality follows from the symmetry of $G_D$.

	Recall that $\nu(x) = \pA_\alpha|x|^{-1-\alpha}$ is the density of the L\'evy measure.	If $A$ is any Borel subset of $(\overline{D})^c$ and $0\le a< b \le \infty$, then we have the following Ikeda-Watanabe formula (see \cite{MR142153})
	\begin{equation*} 
		\mathbb{P}^{x}(\tau_D \in (a,b), X_{\tau_{D}} \in A) = \int_{A} \int _{D} \int_{a}^{b} p_D(t,x,y) \nu(z-y)dtdydz,
	\end{equation*} 
	for the joint distribution of the random vector $(\tau_D, X_{\tau_D})$.
	Taking $a =0$ and $b= \infty$, we get
	\begin{equation} \label{Ikeda-Wantanbe}
		\mathbb{P}^{x}(X_{\tau_{D}} \in A) = \int_{A} \int _{D} G_D(x,y) \nu(z-y)dydz.
	\end{equation}
	Hence, the distribution of $X_{\tau_D}$ is absolutely continuous with respect to the Lebesgue measure. Equation  (\ref{Ikeda-Wantanbe}) is called also Ikeda-Watanabe formula. The density of the measure $\mu(A) = \mathbb{P}^{x}(X_{\tau_{B}} \in A) $ with respect to the Lebesgue'a measure is called Poisson kernel and we denote it by $P_D(x,y)$. For $y\in D$, $P_D(x,y) = 0.$ Hence, 
	\begin{equation*} 
		P_D(x,z) = \int_{D}  G_D(x,y) \nu(z-y)dy, \qquad x\in D,\, z \in (\overline{D})^c.
	\end{equation*}
	
	\begin{lemma} \label{l:nu}
		For $\alpha \in (0,2)$ and $\beta \in (0,1)$,
		\begin{equation*} 
			\int_{D} \nu(1+z) z^{-\beta  + \alpha/2}dz = \overline{C}_{\alpha,\beta}, 
		\end{equation*}
		where 
		\begin{equation*} 
			\overline{C}_{\alpha, \beta} = \mathcal{A}_\alpha B(1-\beta + \alpha/2, \alpha/2 + \beta).
		\end{equation*}
	\end{lemma}
	\begin{proof}
		We have
		\begin{align*}
			\int_D \nu(1+z) z^{-\beta  + \alpha/2} dz &= \mathcal{A}_\alpha \int_0^{\infty} (1+z)^{-1-\alpha}z^{-\beta + \alpha/2} dz = \mathcal{A}_\alpha B(1-\beta + \alpha/2, \alpha/2 + \beta). 	
		\end{align*}
	\end{proof}

	
	\section{Proofs of Theorems \ref{tw:qbeta} and \ref{thm:1}}\label{sec:3}
	

	As it was mentioned in the Introduction, we cannot directly follow the ideas from \cite{MR3460023} to calculate the constant $\kappa_\beta$. Therefore, we first prove the auxiliary result stated in the Theorem \ref{thm:1}.
	
	\subsection{Proof of Theorem \ref{thm:1}}

	First, we will present the general idea of the proof. By the definition of the Green function \eqref{def:GreenF}, we need to prove
\begin{equation} \label{thm:G_D}
		\int_D G_D(x,y) y^{-\beta - \alpha/2}dy = \kappa^{-1}_{ \beta } x^{\alpha/2 - \beta}.
	\end{equation}
	One may try to use here the explicit formula for $G_D(x,y)$ (see e.g. \cite{MR3888401}), however the resulting double integral seems difficult to calculate. Our method proposed below can be also used in the cases where the formula for the Green's function is unknown.  
	In the equation \eqref{thm:G_D} for $G_D(x,y)$ we substitute \eqref{eq:GreenHunt} and in result we obtain that 
	\begin{align*}
		\int_D G_D(x,y) y^{-\alpha/2 - \beta} dy &= \int_D K_\alpha(x,y) y^{-\alpha/2 - \beta} dy + \int_D \mathbb{E}^{x}[K_\alpha(X_{\tau_{D}},y)] y^{-\alpha/2 - \beta} dy \\
		&=C_1 x^{\alpha/2 - \beta} - C_2 \int_D G_D(x,y)y^{-\beta - \alpha/2}dy,
	\end{align*}
	where the constants $C_1, C_2$ can be computed explicitly. Then,  $\kappa_\beta = (1+C_2)/C_1$. 

	However, following this approach we encounter several problems with the convergence of the relevant integrals. Therefore, depending on the range of $\alpha$ and $\beta$, we use appropriate compensation of the potential $K_\alpha$. Consequently, in the proof of the Theorem \ref{thm:1} we will consider separately the cases $\alpha<1$, $\alpha>1$, $\alpha=1$ and the different ranges for $\beta$.

	Note that to prove Theorem \ref{thm:1}, by scaling property of $p_D$ it suffice to show (\ref{eq:1}) only for $x=1$. 
	However, before we pass to the proof of Theorem \ref{thm:1}, we prove a few auxiliary lemmas.

	\begin{lemma} \label{l:op_int_pD_y}
		Let $\alpha \in (0,2)$ and $-\alpha < \gamma < 1 + \alpha/2$. Then,
		\begin{align*}
			\int_D p_D(t,1,y) y^{-\gamma}dy &\approx 1 && \mbox{for}\;\; t < 2^{-\alpha}, &&\\
			\int_D p_D(t,1,y) y^{-\gamma}dy &\approx t^{ -1/2 - \gamma/\alpha }&& \mbox{for}\;\; t>2.&&
		\end{align*} 
	\end{lemma}
	\begin{proof}
		By (\ref{eq:oppt}),  
\begin{equation*}
	p_D(t,1,y)y^{-\gamma} \approx \bigg( 1 \wedge \frac{1}{ t^{1/2} } \bigg) \bigg( 1 \wedge \frac{ y^{\alpha/2} }{ t^{1/2} } \bigg) \bigg( t^{-1/\alpha} \wedge \frac{ t }{ |y - 1|^{1 + \alpha} } \bigg)y^{-\gamma}.
\end{equation*}
For $t < 2^{-\alpha}$ we have 
\begin{align*}
	\int_D p_D(t,1,y) y^{-\gamma}dy & \le c\int_{ 0 }^{ 1/2 } y^{\alpha/2 - \gamma } t^{1/2}  dy + c\int_{1/2}^{3/2}  p(t,x,y) dy + c\int_{ 3/2 }^{ \infty }  t  y^{ -1 -\alpha -\gamma} dy \le c_1, \\
	\int_D p_D(t,1,y) y^{-\gamma}dy &\ge  c_2 \int_{ 1 - t^{ 1/\alpha } }^{ 1 + t^{ 1/\alpha } } t^{ -1/\alpha } dy = 2c_2.			
\end{align*}
For $t > 2$ we have 
\begin{align*}
	\int_D p_D(t,1,y) y^{-\gamma}dy &\approx \int_{ 0 }^{ t^{1/\alpha} }  t^{- 1 - 1/\alpha  } y^{\alpha/2 - \gamma } dy + \int_{ t^{1/\alpha} }^{ \infty } t^{-1/2} \frac{ t }{ y^{ 1 + \alpha } } y^{ -\gamma } dy \\ &= \frac{ t^{- \gamma/\alpha -1/2 } }{ \alpha/2 - \gamma + 1 } + \frac{ t^{ -\gamma/\alpha - 1/2 } }{ \gamma + \alpha } \approx t^{ -1/2 - \gamma/\alpha }.
\end{align*}
\end{proof}

	\begin{corollary} \label{c:op_int_pD_y}
		For $\alpha \in (0,2)$ and $\beta \in (0,1)$ the integral 
		\begin{equation*}
			\int_{0}^{\infty} \int_{D} p_D(t,1,y) y^{-\alpha/2 - \beta} dydt
		\end{equation*}
		is convergent.
	\end{corollary}
	\begin{proof}
		By Lemma \ref{l:op_int_pD_y} we have 
		\begin{align*}
			\int_{0}^{\infty} \int_{D} p_D(t,1,y) y^{-\alpha/2 - \beta} dydt &\leq c_1 \int_{0}^{ 2^{-\alpha} } dt + \int_{ 2^{-\alpha} }^{2}  \int_{0}^{\infty}  p_D(t,x,y)y^{-\alpha/2 - \beta} dydt \\ &+  c_2\int_{2}^{\infty} t^{-1 - \beta/\alpha} dt \\&\leq c + c_3 \int_{0}^{2} y^{-\beta} dy + c_4 \int_{2}^{\infty} \frac{ y^{-\alpha/2 -\beta} }{ |y-1|^{1+\alpha} } dy < \infty.
		\end{align*}
	\end{proof}
	We also note that by (\ref{eq:prob_tau_D}) and Lemma \ref{l:op_int_pD_y} with $\gamma = 0$ we get 
	\begin{equation}\label{eq:lim_prob_tau_D}
		\lim_{t \rightarrow \infty} \mathbb{P}^x(\tau_D > t) = 0, \qquad x>0.
	\end{equation}

	The integral in Corollary \ref{c:op_int_pD_y} is symmetric with respect to argument $\beta$. This is provided by the following lemma.
	\begin{lemma} \label{l:int_pD_symmetry}
		For $x >0$, $\alpha \in (0,2)$ and $\beta \in (0, 1)$ we have
		\begin{equation} \label{eq:int_pD_symmetry}
			\int_{0}^{\infty} \int_{D} p_D(t,x,y) y^{-\beta-\alpha/2} dydt = x^{\alpha/2 - \beta} \int_{0}^{\infty} \int_{D} p_D(t,1,y) y^{-(1 - \beta)-\alpha/2} dydt. 
		\end{equation}
	\end{lemma}
	\begin{proof}
		Since $1-\beta \in (0, 1)$, by Corollary \ref{c:op_int_pD_y} integrals in the formula (\ref{eq:int_pD_symmetry}) are convergent.  Using scaling property of $p_D(t,x,y)$ and letting $s = ty^{-\alpha}$, and then putting $z = x/y$ we have 
		\begin{align*}
			\int_{0}^{\infty} \int_{D} p_D(t,x,y) y^{-\beta-\alpha/2} dydt &= \int_{0}^{\infty} \int_D  y^{-1}p_D(ty^{-\alpha},x/y,1) y^{-\beta - \alpha/2} dy dt \\ &= \int_{0}^{\infty}\int_D p_D(s,x/y,1)y^{-\beta + \alpha/2 -1} dy ds \\ &= x^{ \alpha/2-\beta} \int_{0}^{\infty} \int_D p_D(s, 1, z) z^{-(1-\beta) - \alpha/2} dzds.
		\end{align*}
	\end{proof}

	\smallskip At the beginning, we consider the case $\beta = \alpha/2.$
	\begin{lemma} \label{l:beta=alpha/2}
		For $\alpha \in (0,2)$
		\begin{equation*}
			\int_{ 0 }^{ \infty } \int_{ D } p_D(t,x,y) y^{-\alpha} dy dt =  \frac{ \pi }{\Gamma ( \alpha ) \sin(\pi \alpha/2) }.
		\end{equation*}
	\end{lemma}
	\begin{proof}
		From the Ikeda-Watanabe formula, we have
		\begin{align*}
			\mathbb{P}^x(\tau_D < t ) &= \mathcal{A}_{\alpha} \int_{0}^{ t } \int_D \int_{ D^c } p_D(t,x,y) |y-z|^{-1-\alpha} dzdyds \\ &= \frac{ \mathcal{A}_{\alpha} }{ \alpha } \int_{0}^{ t } \int_D p_D(t,x,y) y^{-\alpha} dyds.
		\end{align*}
		By \eqref{eq:lim_prob_tau_D}, $\lim_{t \rightarrow \infty} \mathbb{P}^{x}(\tau_D < t ) = 1$. Hence, letting $t\to\infty$, we get
		\begin{equation*}
			\int_{0}^{ \infty } \int_D p_D(t,x,y) y^{-\alpha} dyds =  \frac{ \pi }{\Gamma ( \alpha ) \sin(\pi \alpha/2) }.
		\end{equation*}
	\end{proof}
	Now, we will prove 
	\begin{proposal}\label{prop:5C} For $\alpha \in(0,1)\cup(1,2)$ and $\beta \in (0,1/2]$, we have
		\begin{equation}\label{eq:5C} 
			\int_{D} G_D(x,y) y^{-\beta-\alpha/2} dy =  \frac {C_{\alpha} \widehat{C}_{\alpha, \beta} }{  1 + C_{\alpha} \widetilde{C}_{\alpha, \beta} \overline{C}_{\alpha, \beta} }x^{\alpha/2- \beta}, \qquad x \in D,
		\end{equation}
		where $C_{\alpha}$ is the constant from (\ref{C_alpha}), $\widehat{C}_{\alpha, \beta}, \widetilde{C}_{\alpha, \beta}$ are given by  (\ref{const.1}), (\ref{const.2}) respectively, and $\overline{C}_{\alpha, \beta}$ is the constant from Lemma \ref{l:nu}. 
	\end{proposal}
	\begin{proof}
		Let  $\beta \in (0, 1-\alpha/2) \setminus \{\alpha/2\}$. For $z\in \R$ and $y \in D$, we let 
		\begin{align*}
			U(z,y) = 
			\begin{cases}
				C_\alpha |y-z|^{\alpha-1} & \mbox{\;if \;} \alpha<1, \beta \in (\alpha/2,1-\alpha/2), \\
				C_\alpha(|y-z|^{\alpha-1} - y^{\alpha-1}) & \mbox{\;if \;} \alpha<1, \beta \in (0,\alpha/2) \mbox{\;or\;} \alpha>1, \beta \in (0,1-\alpha/2).
			\end{cases}
		\end{align*}
		We will use Facts \ref{f.2} and \ref{f.1} depending on the range of $\alpha$ and $\beta$. By \eqref{eq:GreenHunt} and  \eqref{eq1:f.2} or \eqref{eq1:f.1}, we have  
		\begin{align*}
			\int_{D} G_D(x,y) y^{-\beta-\alpha/2} dy = C_{\alpha} \Big( \widehat{C}_{\alpha,\beta}x^{\alpha/2 - \beta} - \int_D \mathbb{E}^{x}[ U(X_{\tau_D},y)] y^{-\beta - \alpha/2} dy \Big).
		\end{align*}	
		Next, by Ikeda-Watanabe formula, (\ref{eq:beta_def_2}) and \eqref{eq2:f.2} or \eqref{eq2:f.1},
		\begin{align*}
			\int_D \mathbb{E}^{x} [U(X_{\tau_D}, y)] y^{-\beta - \alpha/2} dy 
			&= \int_{D^c} \int_D G_D(x,w)  \nu(w-z) \bigg[ \int_D U(z,y) y^{-\beta - \alpha/2} dy \bigg] dwdz \\
			& = \int_{D} \int_D G_D(x,w)  \nu(w+z) \bigg[ \int_D U(-z,y) y^{-\beta - \alpha/2} dy \bigg] dwdz \\			
			&= \widetilde{C}_{\alpha, \beta} \int_{D} \int_D G_D(x,w) \nu(w+z) z^{-\beta  + \alpha/2} dwdz,
		\end{align*}
		where in the last equality we substituted	$y = z \xi$. 
		By Lemma \ref{l:nu}, 
		\begin{align*}
			\int_{D} \nu(w+z) z^{-\beta  + \alpha/2} dz =  w^{-\beta - \alpha/2} \int_D \nu(1+y)y^{-\beta + \alpha/2} dy  = \overline{C}_{\alpha, \beta}  w^{-\beta - \alpha/2}.
		\end{align*}
		Hence, we obtain
		\begin{align*}
			\int_{D} G_D(x,y) y^{-\beta-\alpha/2} dy = C_{\alpha} \widehat{C}_{\alpha,\beta}x^{\alpha/2 - \beta} - C_{\alpha}\widetilde{C}_{\alpha, \beta} \overline{C}_{\alpha, \beta} \int_D G_D(x,w) w^{-\beta  - \alpha/2} dw,
		\end{align*}
		which gives \eqref{eq:5C}.
		
Next, consider $\alpha>1$ and $\beta \in (1-\alpha/2, 1/2]$. Here,  due to the lack of convergence of cerain integrals  we need to be more subtle. By scaling and symmetry of the Green function, we have 
		\begin{align*}
			G_D(x,y) &=   x^{\alpha - 1} C_{\alpha} \big( |1-y/x|^{\alpha - 1} - \mathbb{E}^{y/x} |X_{\tau_D} - 1|^{\alpha - 1} \big)\\
			&=   x^{\alpha - 1} C_{\alpha} \Big( \big(|1-y/x|^{\alpha - 1} -1 \big) - \big(\mathbb{E}^{y/x} |X_{\tau_D} - 1|^{\alpha - 1}-1 \big) \Big).
		\end{align*}
		Hence, by Fact \ref{f.4},
		\begin{align*}
			\int_D G_D(x,y) y^{-\beta - \alpha/2} dy &=	C_{\alpha} x^{\alpha- 1} \int_{D} \Big( \big( |1-y/x|^{\alpha - 1} - 1 \big) - \big( \mathbb{E}^{y/x} |X_{\tau_D} - 1|^{\alpha-1} - 1 \big) \Big)y^{-\beta - \alpha/2} dy \\
			& = C_{\alpha} x^{\alpha/2 - \beta } \int_{D}  \Big( \big( |1-w|^{\alpha - 1} - 1 \big) - \big( \mathbb{E}^{w} |X_{\tau_D} - 1|^{\alpha-1} - 1 \big) \Big) w^{-\beta - \alpha/2} dw \\
			&= C_{\alpha} x^{\alpha/2 - \beta} \Big(\widehat{C}_{\alpha,\beta} - \int_{D} \big( \mathbb{E}^{w} |X_{\tau_D} - 1|^{\alpha - 1} - 1 \big) w^{-\beta - \alpha/2} dw \Big) =  \mathcal{S} x^{\alpha/2 - \beta},
		\end{align*}
		where $\mathcal{S} = C_{\alpha} \Big(\widehat{C}_{\alpha,\beta} - \int_{D} \big( \mathbb{E}^{w} |X_{\tau_D} - 1|^{\alpha-1} - 1 \big) w^{-\beta - \alpha/2} dw \Big)$. Our aim is to calculate $\mathcal{S}$. 
		By Ikeda-Watanabe formula, Fubini's theorem, symmetry of the Green function, Fact \ref{f.4} and  Lemma \ref{l:nu}, we get 
		\begin{align*}
			&\int_{D}(\mathbb{E}^{w} |X_{\tau_D} - 1|^{\alpha - 1} - 1 )w^{-\beta - \alpha/2} dw \\
			&= \int_{D^c} \int_{D} \bigg[ \int_{D} G_D(w,y) w^{-\beta - \alpha/2}  dw\bigg]\nu(y-z)\big[ |1-z|^{\alpha-1} - 1 \big]dydz \\
			&= \int_{D^c} \int_{D} \mathcal{S} y^{\alpha/2-\beta} \nu(y-z)\big[ |1-z|^{\alpha-1} - 1 \big]dydz \\
			&= \mathcal{S} \mathcal{A}_{\alpha} \int_{D} \int_{D} y^{\alpha/2 - \beta} |y+w|^{-1 - \alpha} \big[ (1+w)^{\alpha-1} - 1 \big]dydw \\
			&= \mathcal{S}\mathcal{A}_{\alpha} \int_{D} \int_{D} y^{\alpha/2 - \beta} (1+y)^{-1-\alpha} \big[(1+w)^{\alpha -1} - 1 \big] w^{-\beta - \alpha/2} dy dw \\
			&= \mathcal{S} \mathcal{A}_{\alpha} \widetilde{C}_{\alpha, \beta} B(1-\beta + \alpha/2, \alpha/2 + \beta) = \mathcal{S} \widetilde{C}_{\alpha, \beta} \overline{C}_{\alpha,\beta}.
		\end{align*}
		Hence, $\mathcal{S} = C_{\alpha} \widehat{C}_{\alpha,\beta} - \mathcal{S} C_\alpha\widetilde{C}_{\alpha, \beta} \overline{C}_{\alpha,\beta}$,  which ends  the proof.
	\end{proof}
	
	We are ready to prove Theorem \ref{thm:1}.
	\begin{proof}[Proof of Theorem \ref{thm:1}]
		First note that for $\beta = \alpha/2$ the assertion follows by Lemma \ref{l:beta=alpha/2}.		Below we will use the identities
		\begin{align}
			\Gamma(1-x) \Gamma(x) &= \frac{\pi}{\sin(\pi x)}, \label{gamma_id.1} \\
			\Gamma(1-x)\Gamma(1+x) &= \frac{ \pi x}{\sin(\pi x)}. \label{gamma_id.2}
		\end{align}
		We will separately consider the cases $\alpha\not=1$ and $\alpha=1$.\\[4pt]
		{\bf  Case $\pmb{\alpha \neq 1}$.} By Proposition \ref{prop:5C} and Lemma \ref{l:int_pD_symmetry} we only need to prove that for $\beta \in (0,1/2] \setminus \{\alpha/2\}$,
		\begin{equation*}
			\frac {C_{\alpha} \widehat{C}_{\alpha, \beta} }{  1 + C_{\alpha} 	\widetilde{C}_{\alpha, \beta} \overline{C}_{\alpha, \beta} } = \frac{\pi }{\Gamma \big(\beta + \alpha/2 \big)\Gamma \big(  1 - \beta + \alpha/2 \big)\sin(\pi \beta) }.
		\end{equation*}
		Note that
		\begin{align*}
			C_{\alpha} \widehat{C}_{\alpha, \beta} = \frac{ 1 }{ 2\Gamma(\alpha)\cos(\pi\alpha/2) } \bigg( \frac{ \Gamma(\alpha)\Gamma(1-\beta -\alpha/2) }{ \Gamma(1-\beta + \alpha/2) } + \frac{\Gamma(\alpha) \Gamma(\beta - \alpha/2)}{\Gamma(\beta + \alpha/2)} \bigg).
		\end{align*}
		Now, applying (\ref{gamma_id.1}) for $x = \beta \pm \alpha/2$ we get
		\begin{align*}
			C_{\alpha} \widehat{C}_{\alpha, \beta} = \frac{ \Gamma(\beta - \alpha/2) }{ \Gamma(\beta + \alpha/2) }  \frac{ \sin(\pi\beta) }{ \sin(\pi(\beta + \alpha/2)) }.
		\end{align*}
		By (\ref{gamma_id.1}) and (\ref{gamma_id.2}), we have
		\begin{align*}
			C_{\alpha} \widetilde{C}_{\alpha, \beta} \overline{C}_{\alpha, \beta} &= 
			\frac{1}{ 2\Gamma(\alpha)\cos(\pi\alpha/2) } B(1-\beta-\alpha/2,\beta-\alpha/2)\frac{\alpha \Gamma(\alpha)\sin(\pi\alpha/2)}{\pi}B(1-\beta+\alpha/2,\beta+\alpha/2) \\
			&=\frac{\tan(\pi\alpha/2)}{2}\frac{\Gamma(1-\beta-\alpha/2)\Gamma(\beta+\alpha/2)}{\pi}\frac{\Gamma(1-\beta+\alpha/2)\Gamma(\beta-\alpha/2)}{\pi}\frac{\alpha\pi}{\Gamma(1-\alpha)\Gamma(1+\alpha)}\\ 
			&=\frac{\tan(\pi \alpha/2) \sin(\pi \alpha) }{ 2\sin(\pi(\beta + \alpha/2)) \sin(\pi(\beta - \alpha/2)) } = \frac{1 - \cos(\pi \alpha) }{ \cos(\pi\alpha) - \cos(2\pi\beta) } .
		\end{align*}
		Therefore,
		\begin{align*}
			1 + C_{\alpha} \widetilde{C}_{\alpha, \beta} \overline{C}_{\alpha, \beta} =  \frac{1 - \cos(2\pi \beta) }{ \cos(\pi\alpha) - \cos(2\pi\beta) } = \frac{ \sin^2(\pi\beta) }{ \sin(\pi(\beta + \alpha/2)) \sin(\pi(\beta - \alpha/2)) }.
		\end{align*}
		Hence, from the above transformations 
		\begin{equation*}
			\frac {C_{\alpha} \widehat{C}_{\alpha, \beta} }{  1 + C_{\alpha} \widetilde{C}_{\alpha, \beta} \overline{C}_{\alpha, \beta} } =  \frac{\pi }{\Gamma \big(\beta + \alpha/2 \big)\Gamma \big(  1 - \beta + \alpha/2 \big)\sin(\pi \beta) }.
		\end{equation*}
		
		\noindent {\bf  Case $\pmb{\alpha = 1}$.} We only need to consider $\beta \in (0, 1/2)$.
		By \eqref{eq:GreenHunt}, we get
		\begin{equation*}
			G_D(x,y) =  \frac{ 1 }{ \pi } \mathbb{E}^{x} \big[ \ln(|X_{\tau_D}-y|/y) \big] - \frac{1}{\pi}\ln(|x-y|/y).
		\end{equation*}
		We integrate both sides with respect to $y^{-\beta - 1/2} dy$. By Fact \ref{f.5}, we obtain 
		\begin{equation*}
			\frac{1}{\pi}\int_{D} \ln(|x-y|/y) y^{-\beta - 1/2} dy = \frac{x^{1/2 - \beta}}{\pi}\int_{D} \ln(|1-y|/y) y^{-\beta - 1/2} dy = \frac{ \sin(\pi \beta ) }{ (1/2 - \beta)\cos( \pi \beta ) } x^{1/2 - \beta}.
		\end{equation*}
		By Ikeda-Watanabe formula and Fact \ref{f.5},
		\begin{align*}
			&\int_{D} \mathbb{E}^{x} \big[\ln \big| X_{\tau_D} - y \big|/y \big]y^{-\beta - 1/2}dy = \int_{D}  \int_{D^c} G_D(x,w)\nu(w-z) \bigg[ \int_D \ln(|y-z|/y)y^{-\beta - 1/2} dy \bigg]dzdw \\&= \int_{D}  \int_{D} G_D(x,w)\nu(w+z)z^{-\beta + 1/2} \bigg[ \int_D \ln(1+u)u^{(\beta - 1/2) - 1} du \bigg]dzdw \\&= \frac{\pi}{ (1/2 - \beta) \cos(\pi \beta) } \int_{D}  \int_{D} G_D(x,w)\nu(w+z)z^{-\beta + 1/2}dzdw \\&= \frac{\pi}{ (1/2 - \beta) \cos(\pi \beta) } \overline{C}_{1, \beta} \int_D G_D(x,w)w^{-\beta - 1/2} dw.
		\end{align*}
		Hence, we get 
		\begin{align*}
			\int_{D} G_D(x,y) y^{-\beta-1/2} dy &= \frac{ \sin(\pi \beta ) x^{1/2 -\beta} }{ (1/2 - \beta)\cos( \pi \beta ) } \bigg( -1 + \frac{  B(3/2-\beta, 1/2 + \beta) }{ \pi(1/2 - \beta) \cos(\pi \beta) } \bigg)^{-1} \\ &= \frac{ \sin(\pi \beta ) x^{1/2 -\beta} }{ (1/2 - \beta)\cos( \pi \beta ) } \bigg( -1 + \frac{  1 }{ \cos^2(\pi \beta) } \bigg)^{-1} \\
			&= \frac{ \sin(\pi \beta )x^{1/2 -\beta} }{ (1/2 - \beta)\cos( \pi \beta ) } \frac{ \cos^2(\pi \beta) }{ \sin^2(\pi \beta) } = \frac{\pi x^{1/2 -\beta} }{\Gamma \big(\beta + 1/2 \big)\Gamma \big(  1 - \beta + 1/2 \big)\sin(\pi \beta) }.
		\end{align*}
	\end{proof}
	
	\subsection{Proof of Theorem \ref{tw:qbeta}}
	Recall that 
	\begin{align*}
		f(t) &= 	\mathcal{C}^{-1} t^{ (-\alpha/2 - \beta + \gamma)/\alpha } \mathbbm{1}_{(0,\infty)}(t),\\
		h_{\beta}(x) &= \int_{0}^{\infty} \int_{D} p_D(t,x,y)f(t)y^{-\gamma} dydt,\qquad x \in D,\\
		q_{\beta}(x) &= \frac{1}{h_{\beta}(x)} \int_{0}^{\infty} \int_{D} p_D(t,x,y)f'(t)y^{-\gamma} dydt,\qquad x \in D.
	\end{align*}
	where the constant $\mathcal{C}$ is given by \eqref{eq:constC}. First, we prove that: 
	\begin{lemma} \label{l:hbeta}
		For $\alpha \in (0,2),~\beta \in (0,1)$ we have
		\begin{equation*}
			h_{\beta}(x) = x^{\alpha/2 - \beta},\qquad x \in D.
		\end{equation*}
	\end{lemma}
	\begin{proof}
		We use similar argument as in Lemma \ref{l:int_pD_symmetry}. By Lemma \ref{l:op_int_pD_y} the integral 
		\begin{equation*}
			\int_{0}^{\infty} \int_{D} p_D(t,x,y)f(t)y^{-\gamma} dydt
		\end{equation*}
		is finite. 
		By (\ref{eq:sca}) and \eqref{eq:constC}
		\begin{align*}
			&\int_{0}^{\infty} \int_{D} p_D(t,x,y)f(t)y^{-\gamma} dydt = \int_{0}^{\infty} \int_{D}  x^{-1} p_D(tx^{-\alpha}, 1, y/x) f(t)y^{-\gamma} dydt \\&= \int_{0}^{\infty} \int_{D} x^{\alpha - 1}p_D(s,1,y/x) x^{-\alpha/2 - \beta + \gamma} f(s) y^{-\gamma} dyds \\&= \int_{0}^{\infty} \int_{D} x^{\alpha/2 - 1 - \beta + \gamma }p_D(s,1,z) f(s) x^{-\gamma+1} z^{-\gamma} dzds = x^{\alpha/2 - \beta}.
		\end{align*}
	\end{proof}

	\begin{proof}[Proof of Theorem \ref{tw:qbeta}]
		By Lemma \ref{l:op_int_pD_y} the integral 
		\begin{equation*}
			\int_{0}^{\infty} \int_{D} p_D(t,x,y)f'(t)y^{-\gamma} dydt
		\end{equation*}
		is finite. By (\ref{eq:sca}) and using similar argument as in Lemma \ref{l:hbeta} we get
		\begin{align}\label{eq:q1}
		h_\beta(x)	q_\beta(x) = 	x^{-\alpha/2-\beta} \int_{0}^{\infty} \int_{D} p_D(s,1,z) f'(s) z^{-\gamma} dzds = x^{-\alpha/2-\beta} q_\beta(1).
		\end{align}
		It remains to show that $q_\beta(1) = \kappa_{\beta}$. By Lemma \ref{l:hbeta}, \eqref{eq:q1} and Fubini theorem, we have 
		\begin{align*}
			1 =  h_{\beta}(1) &= \int_{0}^{\infty} \int_{D} p_D(t,1,y)f(t)y^{-\gamma} dydt 
			= \int_{0}^{\infty} \int_{D} \int_{0}^{t} p_D(t,1,y)f'(s)y^{-\gamma} dsdydt \\
			&=  \int_{0}^{\infty} \int_{D} \int_{0}^{\infty} p_D(t+s,1,y)f'(s)y^{-\gamma} dtdyds \\ 
			&=  \int_{0}^{\infty} \int_{D} \int_{0}^{\infty} \int_{D} p_D(t,1,w) p_D(s,w,y) f'(s)y^{-\gamma} dydsdwdt \\
			&=  \int_{0}^{\infty} \int_{D}  p_D(t,1,w) h_\beta(w) q_\beta(w) dwdt \\
			&= q_{\beta}(1) \int_{0}^{\infty} \int_{D} p_D(t,1,w)w^{-\alpha/2 - \beta} dwdt.
		\end{align*}
		Hence, by Theorem \ref{thm:1} we get $q_{\beta}(1) = \kappa_{\beta}$.	
	\end{proof}

	By Lemma \ref{l:int_pD_symmetry}, $\kappa_{\beta} = \kappa_{1-\beta}$. Now, we will complete this chapter by showing that $\kappa_\beta$ is increasing for $\beta \in(0,1/2)$ and decreasing when $\beta \in(1/2,1)$ and reaches its maximum at $\beta = 1/2$ (see Figure \ref{fig:k}). We will follow the idea of the proof from \cite{MR3460023}. We have 
	\begin{equation*}
		\frac{\Gamma'(x)}{\Gamma(x)} = -\gamma - \sum_{k=0}^{\infty} \Big( \frac{1}{x+k} - \frac{1}{k + 1} \Big),
	\end{equation*}
	where $\gamma$ is the Euler-Mascheroni constant. Taking derivative of $\log(\kappa_\beta)$ we get
	\begin{align*}
		\frac{\kappa'_{ \beta }}{\kappa_{ \beta }} &= - \sum_{k=0}^{\infty} \bigg( \frac{1}{\beta + \alpha/2 + k} - \frac{1}{ 1- \beta + \alpha/2 + k } - \frac{1}{ \beta + k } + \frac{1}{ 1 - \beta + k } \bigg) \\ &= \sum_{k=0}^{\infty} \bigg( \frac{ 2\beta - 1 }{ (\beta + \alpha/2 + k)(1- \beta + \alpha/2 + k) } - \frac{ 2\beta - 1 }{ (\beta + k) (1 - \beta + k) } \bigg) \\&= \sum_{k=0}^{\infty} \frac{ (1-2\beta) \alpha/2(2k + 1 + \alpha/2) }{ (\beta + \alpha/2 + k)(1- \beta + \alpha/2 + k)(\beta + k) (1 - \beta + k) }.
	\end{align*}
	Hence $\kappa'_{ \beta } > 0$ for $\beta \in (0, 1/2)$,  $\kappa'_{ \beta } = 0$ when $\beta = 1/2$ and $\kappa'_{ \beta } < 0$ when $\beta \in (1/2,1)$. 
	\begin{figure}[htp] 
		\centering
				\includegraphics[scale=0.4]{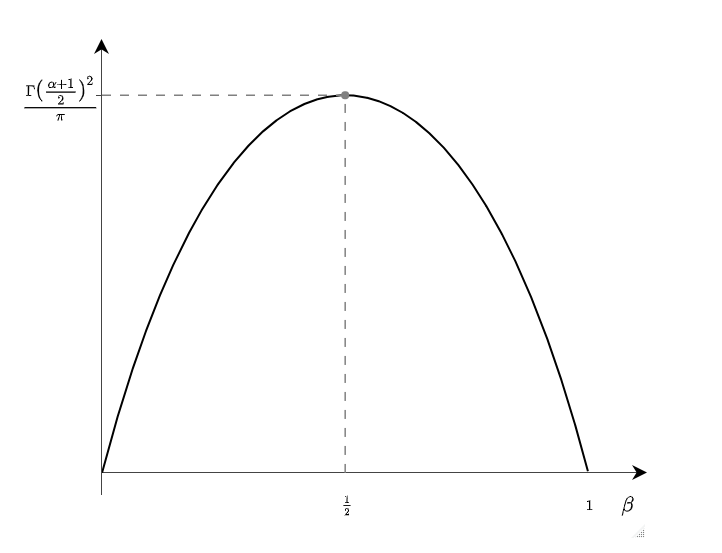}
		\caption{The function $\beta \rightarrow \kappa_{ \beta }$.} 
		\label{fig:k}
	\end{figure}
	
	\begin{remark}\label{rem:1} \rm
		One may try to prove Theorem \ref{tw:qbeta} by following more directly the idea from \cite{MR3933622} and calculate the integral
		\begin{align*}
			\int_0^\infty \int_D t^{\beta} p_D(t,x,y) dy dt  = \int_0^\infty t^{\beta} \P^x(\tau_D >t) dt = \frac{1}{\beta+1}\E^x( \tau_D^{\beta+1}).
		\end{align*}
		We may proceed as follows. We denote $S_t = \sup_{0\le s \le t} X_s$, then $\P^x(\tau_D >t) = \P^0(S_t <x)$. Following \cite{MR2789582}, we denote by
		$\mathcal{M}(z) = \E^0(S_1^{z-1})$
		the Mellin transform of $S_1$. By \cite[Theorem 8]{MR2789582}, $\mathcal{M}$ may be expressed by a complicated formula involving some special functions (double Gamma functions). Now, using scaling properties and after some calculations we get
		$\mathcal{M}(z) = \E^1(\tau_D^{(1-z)/\alpha})$. Hence,  
		\begin{align*}
			\int_0^\infty \int_D t^{\beta} p_D(t,1,y) dy dt  = \mathcal{M}(1-\alpha(\beta+1)).
		\end{align*}
		It seems that by using \cite[Theorem 7]{MR2789582} it is possible to obtain the assertion of Theorem \ref{thm:1} at least for some range of the parameters $\alpha$ and $\beta$. Like in our approach, however one has to struggle with some integrability issues.  
	\end{remark}
			
	\begin{remark}\label{rem:2} \rm
		Since $p_D(t,x,y) < p(t,x,y)$ and in particular $\lim_{x\to 0} p_D(t,x,y)/p(t,x,y) =0$, it seems natural that $p_D$ may be perturbed by the bigger potential than $p$. We verify below that indeed $\kappa_\beta > \kappa_\beta^\R$, where $\kappa_\beta^\R$ is given by \eqref{eq:kappabRd} with $d=1$.
		\newline
		\noindent Let $0 < \alpha < 1$ and $0 < \beta \leq \frac{1-\alpha}{2}$. The condition $\kappa_{\beta} < \kappa_{\beta}^{\mathbb{R}}$ is equivalent to
		\begin{equation*}
			\frac{ \Gamma \big( \beta \big) \Gamma \big( 1 - \beta \big) \Gamma \big( \frac{1 - \beta}{2} \big) }{ \Gamma \big( \frac{ \beta }{ 2 } \big) } < \frac{ \Gamma \big( \beta + \frac{\alpha}{2} \big) \Gamma \big( 1 - \beta + \frac{\alpha}{2} \big) \Gamma \big( \frac{1 - \beta}{2} - \frac{\alpha}{2} \big) }{ 2^{\alpha} \Gamma \big( \frac{ \beta }{ 2 } + \frac{\alpha}{2} \big) },
		\end{equation*}
		hence it suffice to show that for fixed $\beta \in (0, \frac{1-\alpha}{2})$, $u(t) = \frac{ \Gamma \big( \beta + t \big) \Gamma \big( 1 - \beta + t \big) \Gamma \big( \frac{1 - \beta}{2} - t \big) }{ 2^{\alpha} \Gamma \big( \frac{ \beta }{ 2 } + t \big) } $ is increasing on the interval $[0, \frac{1}{2} - \beta)$. We consider 
		\begin{equation*}
			\big[ \log(u(t)) \big]' =  \sum_{k=0}^{ \infty } \Big( - \frac{ 1 }{ \beta + t + k } - \frac{ 1 }{ 1- \beta + t + k } + \frac{ 1 }{ \frac{ 1- \beta }{ 2 } - t + k} + \frac{ 1 }{ \frac{\beta}{2} + t + k } \Big) - 2\ln(2).
		\end{equation*}
		Note that 
		\begin{equation*}
			\odv{  }{ t } \Big( - \frac{ 1 }{ \beta + t + k } + \frac{ 1 }{ \frac{\beta}{2} + t + k } \Big) < 0 \quad~\text{and}~\quad \odv{}{t} \Big( - \frac{ 1 }{ 1- \beta + t + k } + \frac{ 1 }{ \frac{ 1-\beta }{ 2 } - t + k } \Big) > 0,
		\end{equation*}
		therefore both functions above are strictly monotone. Hence, taking $t = \frac{1}{2} - \beta$ and $t = 0$ respectively, we get
		\begin{equation*}
			\big[ \ln(u(t)) \big]' \geq  \sum_{k=0}^{ \infty } \Big( \frac{2}{ \frac{ 1-\beta }{2} + k } - \frac{1}{ 1- \beta + k } - \frac{1}{ \frac{1}{2} + k } \Big) - 2\ln(2).
		\end{equation*}
		Now we investigate monotonicity with respect to variable $\beta$:
		\begin{equation*}
			\odv{ }{ \beta } \Big( \frac{ 2 }{ \frac{ 1 - \beta }{ 2 } + k } -  \frac{ 1 }{ 1 - \beta + k } - \frac{ 1 }{ \frac{ 1 }{ 2 } + k } \Big) = \frac{ 1 }{ ( \frac{ 1 - \beta }{ 2 } + k )^2 } - \frac{ 1 }{ (1 - \beta + k)^2 } > 0,
		\end{equation*}
		so it attains its minimum at $\beta = 0$. Hence, (see \cite[Equation 8.363.3]{MR2360010})
		\begin{equation*}
			[\ln (u(t))]' \geq  \sum_{k=0}^{\infty} \Big( \frac{ 1 }{ \frac{1}{2} + k } - \frac{ 1 }{ 1 + k } \Big) - 2\ln(2) = 0.
		\end{equation*}
		On Figure \ref{fig:2} we compare $\kappa_{\beta}$ and $\kappa_{\beta}^{\mathbb{R}}$ for $\alpha = \frac{1}{2}.$
		\begin{figure}[htp] 
			\centering
						\includegraphics[scale=0.4]{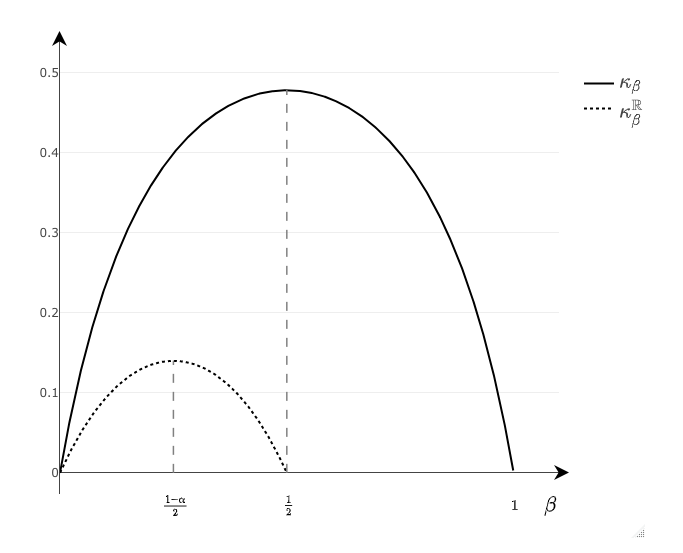}
			\caption{Comparison of $\kappa_{ \beta }$ and $\kappa_{ \beta }^{ \mathbb{R} }$ for $\alpha = \frac{1}{2} $.} 
			\label{fig:2}
		\end{figure}
	\end{remark}
	
	\section{Hardy identity} \label{sec:4}
		
	In this section we prove Hardy identity for transition density $p_D(t,x,y)$. We will use the results from \cite{MR3460023}. Recall that by $P_t^D$ we denote the semigroup generated by the process $X_t$ killed on exiting $D$. We first define the quadratic form $\mathcal{E}$ of $\Delta_D^{\alpha/2}$: 
	\begin{equation*}
		\mathcal{E}_D(u,u) = \lim_{t \to 0+} =\frac{1}{t} (u-P_t^Du,u),\qquad u\in L^2(D).
	\end{equation*}
	Here, as usual $(u,v) = \int_D u(x)v(x)dx$ for $u,v \in L^2(D).$ 
	We let
	\begin{equation*}
		\mathcal{D}(\mathcal{E}_D)=\{u\in L^2(D):\mathcal{E}_D(u,u)<\infty\}.
	\end{equation*}
	Now, we will prove an auxiliary lemma:
	\begin{lemma} \label{lemat_nu_granica}
		Let $x,y \in D.$ We have
		\begin{equation*}
			\lim_{t \rightarrow 0} \frac{p_D(t,x,y)}{t} = \nu(x-y).
		\end{equation*} 
	\end{lemma}
	\begin{proof}
		By Hunt's formula, 
		\begin{equation*}
			\frac{p_D(t,x,y)}{t} = \frac{p(t,x,y)}{t} - \frac{ \mathbb{E}^x[p(t-\tau_D, X_{\tau_{D}}, y)\mathbf{1}_{ \{\tau_D < t \} } ] }{t}.
		\end{equation*}
		Limit of the first term is known (see. e.g. \cite{MR3165234}):
		\begin{equation*}
			\lim_{t \rightarrow 0} \frac{p(t,x,y)}{t} = \nu(x-y).
		\end{equation*}
		Hence it suffices to show that, the second term converges to zero. We have
		\begin{align*}
			\frac{ \mathbb{E}^x[p(t-\tau_D, X_{\tau_{D}}, y)\mathbf{1}_{ \{\tau_D < t \} } ] }{t} \leq c  \mathbb{E}^x \bigg[ \frac{ t - \tau_{D} }{ |X_{\tau_{D}} - y|^{1+\alpha} }\mathbf{1}_{ \{\tau_D < t \} } \bigg] t^{-1} \leq \frac{c}{y^{1+\alpha}} \mathbb{P}^{x} (\tau_D < t).
		\end{align*}
		By Ikeda-Watanabe formula and scaling property
		\begin{align*}
			\mathbb{P}^{x} (\tau_D < t) &= \int_D \int_{D^c} \int_{0}^{t}  p_D(s,x,y) \nu(z-y)dsdydz \\ &=c \int_D \int_{0}^{t} p_D(s,x,y)y^{-\alpha}dsdy = c \int_D \int_{0}^{tx^{-\alpha}}  p_D(s,1,y)y^{-\alpha} dsdy.
		\end{align*}
		For $t$ such that $tx^{-\alpha} < 1$ and by Lemma \ref{l:op_int_pD_y} we get
		\begin{equation*}
			\int_D \int_{0}^{tx^{-\alpha}} p_D(s,1,y)y^{-\alpha} dsdy \leq ct \stackrel{t\to0}{\longrightarrow} 0.
		\end{equation*}
	\end{proof}
	
	\begin{proof}[Proof of Theorem \ref{thm:HardyId}]
		Since assumptions of \cite[Theorem 2]{MR3460023}  are satisfied, we have
		\begin{align*}
			\mathcal{E}_D(u,u) = &\kappa_{ \beta }\int_D \frac{u^2(x)}{x^{\alpha}} dx \\ &+\lim_{t \rightarrow 0} \int_D \int_D \bigg( \frac{u(x)}{x^{\alpha/2 - \beta}} - \frac{u(y)}{y^{\alpha/2 - \beta}} \bigg)^2 x^{\alpha/2 - \beta} y^{\alpha/2 - \beta} \frac{p_D(t,x,y)}{ 2t }dxdy.
		\end{align*}
		By Lemma  \ref{lemat_nu_granica} we get
		\begin{equation*}
			\lim_{t \rightarrow 0} \frac{p_D(t,x,y)}{ t }  = \nu(x-y),
		\end{equation*}
		so we have only to justify that we can change the order of the limit and the integral. Note that $\frac{p_D(t,x,y)}{ t } \leq c\nu(x-y)$. Hence if
		\begin{equation*}
			\int_D \int_D \bigg( \frac{u(x)}{x^{\alpha/2 - \beta}} - \frac{u(y)}{y^{\alpha/2 - \beta}} \bigg)^2 x^{\alpha/2 - \beta} y^{\alpha/2 - \beta}\nu(x-y)dxdy< \infty,
		\end{equation*}
		then we apply dominated convergence theorem. If
		\begin{equation*}
			\int_D \int_D \bigg( \frac{u(x)}{x^{\alpha/2 - \beta}} - \frac{u(y)}{y^{\alpha/2 - \beta}} \bigg)^2 x^{\alpha/2 - \beta} y^{\alpha/2 - \beta}\nu(x-y)dxdy = \infty,
		\end{equation*}
		then we apply Fatou's lemma.
	\end{proof}
	
	\noindent Since $\kappa_{ \beta }$ is increasing for $\beta \in (0,1/2)$ and decreasing when $\beta \in (1/2,1)$, it attains its maximum at $1/2$. Hence, we obtain the following result first obtained in \cite[(1.5)]{MR2663757}.
	
	\begin{proposal}[Hardy inequality]\label{Prop:HardyIn}
		For $u \in L^2(D)$ we have
		\begin{equation}\label{Eq:HardyIn}
			\mathcal{E}_D(u,u) \geq \frac{ \Gamma((\alpha + 1)/2)^2 }{\pi} \int_D \frac{u^2(x)}{ x^{\alpha} } dx.
		\end{equation}
	\end{proposal}
	It turns out that $\kappa_{1/2} = \frac{ \Gamma((\alpha + 1)/2)^2 }{\pi}$ is the best possible constant in the inequality above, see \cite{MR2663757} for details.
		
	\begin{remark}\label{rem:3} \rm
		We note that for $d = 1, \alpha < 1$ and $\beta  < 1 - \alpha$, \eqref{eq:HardyId} may be obtained directly from (\ref{eq:p2}). For $u$ with support in $D$ we put $\beta - \alpha/2$ for $\beta$ in (\ref{eq:p2}) and we get
		\begin{equation*}
			\mathcal{E}(u,u)= 
			(\kappa_{\beta-\alpha/2}^{\R} +  \overline{C}_{\alpha,\beta})
			\int_{D} \frac{u(x)^2}{|x|^\alpha}\,dx
			+ \frac{1}{2} \int_{D}\!\int_{D}
			\left[\frac{u(x)}{x^{\alpha/2-\beta}}-\frac{u(y)}{|y|^{\alpha/2-\beta}}\right]^2
			|x|^{\alpha/2-\beta}|y|^{\alpha/2-\beta} \nu(x,y)
			\,dy\,dx.
		\end{equation*}
	Since $\pE_D$ coincides with $\pE$ for the functions supported in $D$ we get \eqref{eq:HardyId} provided 
	\begin{equation*}
		\kappa^{\mathbb{R}}_{\beta - \alpha/2} + \overline{C}_{\alpha,\beta} = \kappa_{\beta}.
	\end{equation*}
	Indeed, the equality above holds for $\alpha \in (0,2)$ and $\beta \in (0,1)$. By (\ref{gamma_id.1}) and the formula
	\begin{equation*}
		\frac { \Gamma( x ) }{ \Gamma(1/2 - x) } = \frac{ 2^{1-2x} \cos(\pi x) \Gamma( 2x ) }{ \pi^{1/2} },
	\end{equation*}
	we get
	\begin{equation*}
		\kappa^{\mathbb{R}}_{\beta - \alpha/2} =  \frac{1}{\pi} \Gamma(1-\beta + \alpha/2) \Gamma(\beta + \alpha/2) 2 \cos(\pi(\beta/2 + \alpha/4)) \sin(\pi(\beta/2 - \alpha/4)).
	\end{equation*}
	Now, applying identity $\cos(b+a)\sin(b-a) = \frac{1}{2} \big( \sin(2b) - \sin(2a) \big)$,	we get
	\begin{equation*}
		\kappa^{\mathbb{R}}_{\beta - \alpha/2} + \overline{C}_{\alpha,\beta} = \Gamma(1-\beta + \alpha/2) \Gamma(\beta + \alpha/2) \frac{ \sin(\pi\beta) }{\pi} = \kappa_{\beta}.
	\end{equation*}
	
	\end{remark}
	\section*{Acknowledgment}
	We thank Krzysztof Bogdan for many helpful discussions and comments on the paper.
	
	\bibliographystyle{abbrv}
	\bibliography{shlbib}
	
\end{document}